\documentclass[12pt,twoside]{article}
\usepackage{amsmath,amssymb,amsthm,mathrsfs,color,times,textcomp,sectsty,verbatim}
\usepackage[colorlinks=true]{hyperref}
\hypersetup{urlcolor=blue, citecolor=red, linkcolor=blue}
 \allowdisplaybreaks
\begin{document}
\renewcommand{\proofname}{\bf Proof}
\let\oldsection\section
\renewcommand\section{\setcounter{equation}{0}\oldsection}
\renewcommand\thesection{\arabic{section}}
\renewcommand\theequation{\thesection.\arabic{equation}}
\newtheorem{claim}{\indent Claim}[section]
\newtheorem{theorem}{\indent Theorem}[section]
\newtheorem{lemma}{\indent Lemma}[section]
\newtheorem{proposition}{\indent Proposition}[section]
\newtheorem{definition}{\indent Definition}[section]
\newtheorem{remark}{\indent Remark}[section]
\newtheorem{corollary}{\indent Corollary}[section]
\newtheorem{example}{\indent Example}[section]
\title{\Large \bf
The Paneitz-Sobolev constant of a closed Riemannian\\ manifold
and an application to the nonlocal \\$\mathbf{Q}$-curvature flow}

\author{Xuezhang  Chen\thanks{ X. Chen: xuezhangchen@nju.edu.cn}
\\
 \small
$^\ast$Department of Mathematics \& IMS, Nanjing University, Nanjing
210093, P. R. China
}

\date{}
\maketitle
\begin{abstract}
In this paper, we establish that: Suppose a closed Riemannian manifold $(M^n,g_0)$ of dimension $\geq 8$ is not locally conformally flat, then the Paneitz-Sobolev constant of $M^n$ has the property that $q(g_0)<q(S^n)$. The analogy of this result was obtained by T. Aubin in 1976 and had been used to solve the Yamabe problem on closed manifolds. As an application, the above result can be used to recover the sequential convergence of the nonlocal Q-curvature flow on closed manifolds recently introduced by Gursky-Malchiodi. 
\end{abstract}

\section{Introduction and the main result}

\indent \indent Similar to the Yamabe problem, or more generally, to the
prescribing scalar curvature problem on $S^n$, a nature problem on closed manifolds involving the fourth order
conformally covariant operator can be proposed as follows:\\

 On a closed manifold $(M^n,g_0)$ of dimension $\geq 5$,
does there exist a conformal metric $g=u^{4 \over n-4}g_0$ with constant $Q_g$-curvature? \\

The above problem on $M^n$ with $n \geq 5$ is reduced to the
solvability of the following equation
\begin{equation}\label{prescribed_Q-curvature}
P_{g_0}( u)=\tfrac{n-4}{2} c u^{n+4 \over n-4} \text{~~and~~} u>0 \hbox{~~on~~} M^n,
\end{equation}
where $c$ is a constant and $P_{g_0}$ is fourth-order conformally covariant operator on $(M^n,g_{g_0})$, which will be defined below soon. Recently, under the assumptions that $Q_{g_0}$ is semi-positive and the scalar curvature $R_{g_0}$ is nonnegative, an affirmative answer to the above problem is given by Gursky-Malchiodi \cite{gur_mal}. For more background of the above mentioned problem, one may refer to \cite{gur_mal} and the references therein. \\
\indent Let $(M^n,g)$ be a smooth Riemannian manifold  of dimension
larger than four, and $R_g, \text{Ric}_g$ be the scalar curvature,
Ricci curvature of metric $g$, respectively. The following
conformally covariant operator of order four on $(M^n,g)$ is
discovered by T. Branson \cite{branson} and Fefferman-Graham
\cite{feffgra}, concretely,
\begin{eqnarray}
P_g u  &=& \Delta_g^2+\hbox{div}_g\{(4A_g-(n-2)\sigma_1(A_g)g)(\nabla u, \cdot)\}+\frac{n-4}{2}Q_g u \label{P_g}
\end{eqnarray}
and the $Q$-curvature $Q_g$ of metric $g$ is defined by
\begin{eqnarray}\label{def_Q-curv}
Q_g=-{1 \over 2(n-1)}\Delta_g R_g+{n^3-4n^2+16n-16 \over 8(n-1)^2(n-2)^2} R_g^2-{2 \over (n-2)^2}|\text{Ric}_g|^2,
\end{eqnarray}
with the Schouten tensor $A_g=\frac{1}{n-2}(\hbox{Ric}_g-\frac{R_g}{2(n-1)}g)$.
Under  conformal change $\bar{g}=u^{4 \over n-4}g$, there holds
\begin{equation}\label{conformal-invariant}
P_g(\varphi u)=u^{n+4 \over n-4}P_{\bar{g}}(\varphi)
\end{equation}
for all $\varphi \in C^\infty(M^n)$.

We first state the main reuslt of this paper analogous to T. Aubin's \cite{aubin} in the Yamabe problem.
\begin{theorem}\label{Thm1}
On a closed Riemannian manifold $(M^n,g_0)$ of dimension $\geq 8$, suppose there exists $p \in M^n$ such that the Weyl tensor $W_{g_0}(p) \neq 0$, then $q(g_0)<q(S^n)$.
\end{theorem}
We should point out here that some ideas of the proof of Theorem \ref{Thm1} involve the ones in \cite{er}.

The structure of this paper is as follows: In section \ref{sec2}, some standard results to experts in this field are presented. In section \ref{sec3}, we establish the main result of this paper, Theorem \ref{Thm1}. In section \ref{sec4}, the main Theorem \ref{Thm1} is applied to recover the sequential convergence of the nonlocal $\mathbf{Q}$-curvature flow in \cite{gur_mal}.

\begin{remark}
From the comments of Fr$\acute{e}$d$\acute{e}$ric Robert,  the two formulae (21) and (22) on page 511 in \cite{er} may also involve Theorem \ref{Thm1}. For the readers' convenience, one may compare their computations with ours appeared in the proof of Theorem \ref{Thm1}.
\end{remark}

{\bf Acknowledgments:} 
The author is partially supported through NSFC (No.11201223) and Program for New Century Excellent Talents in University (NCET-13-0271). He would thank Dr. Yalong Shi for many stimulating discussions, especially in Lemma 3.1 and Professor Xingwang Xu for helpful comments and precious advice on the nonlocal $\mathbf{Q}$-curvature flow.

\section{Preliminaries}\label{sec2}

\indent \indent The following results in this section are standard for the experts in this field. Let $(M^n,g_0)$ be a closed Riemannian manifold. Define a energy functional on $C^\infty(M^n)$ by
$$E_{g_0}[u]=\int_{M^n} u P_{g_0}(u)d\mu_{g_0},$$
where $P_{g_0}$ is defined as (\ref{P_g}).

\begin{lemma}\label{conf_inv_P_g}
For any conformal metric $g$
of $g_0$, then the Paneitz-Sobolev constant
$$q(g)=q(M^n, g)\equiv\inf\left\{\frac{\int_{M^n}w P_g(w)d\mu_g}
{(\int_{M^n}w^{2n \over n-4}d\mu_g)^{n-4 \over n}}; w \in
C^\infty(M^n)\setminus \{0\}\right\}$$ is independent of the selection of
metrics in the conform class of $g_0$.
\end{lemma}
\begin{proof} For any $g_1,g_2$ in the conformal class of $g_0$, there exists some positive smooth function $\varphi$
such that $g_2=\varphi^{4 \over n-4}g_1$. Thus by the definition of
$q(g_1)$ and (\ref{conformal-invariant}), for any $w \in
C^\infty(M^n)\setminus \{0\}$, there holds
\begin{eqnarray*}
q(g_1) \leq \frac{\int_{M^n}w \varphi P_{g_1}(w
\varphi) d\mu_{g_1}}{(\int_{M^n} (w \varphi)^{2n \over
n-4}d\mu_{g_1})^{n-4 \over n}}=\frac{\int_{M^n}w
P_{g_2}(w)d\mu_{g_2}}{(\int_{M^n}w^{2n \over n-4}d\mu_{g_2})^{n-4
\over n}}.
\end{eqnarray*}
Taking the infimum for all $w \in
C^\infty(M^n)\setminus \{0\}$ over the above inequality to show $q(g_1)\leq q(g_2)$. Exchanging $g_1$ and $g_2$, we also obtain $q(g_2) \leq q(g_2)$. Thus $q(g_1)=q(g_2)$ holds for any conformal metrics $g_1, g_2$ of $g_0$. 
\end{proof}

Next we establish the so-called ``Kazdan-Warner" condition for prescribed $Q$-curvature problem on $S^n$. \\
\indent Let $g=u^{4 \over n-4} g_{S^n}$ with $0<u \in
C^\infty(S^n)$, in particular set $\varphi=1$ in \eqref{conformal-invariant}, there holds
$$P_{S^n}u={n-4 \over 2} Q_g u^{n+4 \over n-4} \text{~~on~~} S^n.$$

Let $\phi$ be a conformal transformation on $S^n$, define its companion of $u$ related
to $\phi$ by
$$v=(u\circ \phi)|\det d\phi|^{n-4 \over 2n}, \text{~~and~~}\phi^\ast(g_{S^n})=|\det d\phi|^{2/n}g_{S^n}.$$

\begin{lemma}\label{kazdan_warner_con}
For any conformal vector field $X$ on $(S^n,g_{S^n})$, there hold $E_{g_{S^n}}[v]=E_{g_{S^n}}[u]$ and
$$\int_{S^n}<X,\nabla Q_g>_{S^n} u^{2n \over n-4} d\mu_{S^n}=0,$$
where $Q_g$ is the $Q$-curvature of the conformal metric $g=u^{4
\over n-4}g_{S^n}$.
\end{lemma}
\begin{proof} By \eqref{conformal-invariant} and variable
change formula, a direct computation yields
\begin{eqnarray*}
E_{g_{S^n}}[v]&=&\int_{S^n}v P_{S^n}v d\mu_{S^n}\\
&=&\int_{S^n} (u\circ \phi)|\det d\phi|^{n-4 \over 2n}
P_{S^n}( (u\circ \phi)|\det d\phi|^{n-4 \over 2n})d\mu_{S^n}\\
&=&\int_{S^n}(u\circ \phi) P_{|\det d\phi|^{2/n}g_{S^n}} (u\circ \phi) (\det d\phi) d\mu_{S^n}\\
&=&\int_{S^n}(u\circ \phi) P_{\phi^\ast(g_{S^n})}(u\circ \phi) d\mu_{\phi^\ast(g_{S^n})}\\
&=&\int_{S^n}u P_{S^n}(u)d\mu_{S^n}=E_{g_{S^n}}[u].
\end{eqnarray*}

For the second assertion, denote by $\phi(t)$ the family of smooth
conformal transformations on $S^n$ induced by the vector field $X$
with $\phi(0)=\text{id}$, and set its corresponding conformal vector
field $\xi(t)=(d\phi(t))^\ast {d \phi \over dt}$. In particular,
$X=\xi(0)$. Define the companion of $u$ relating to $\phi(t)$ by
$$w(t)=(u\circ\phi(t))(\det d\phi)^{n-4 \over 2n} \text{~~or~~} w^{4 \over n-4}g_{S^n}=\phi^\ast(g).$$
From conformal covariance (\ref{conformal-invariant}) of $P_{S^n}$,
$w$ solves
\begin{equation}\label{Q-curvature_w}
P_{S^n}(w)={n-4 \over 2} (Q_g \circ \phi) w^{n+4 \over n-4}
\text{~~on~~} S^n.
\end{equation}
Differentiate (\ref{Q-curvature_w}) with respect to $t$ to yield
\begin{equation}\label{eqn_w_t}
P_{S^n}(w_t)={n-4 \over 2}\Big[\xi \cdot d(Q_g \circ \phi)w^{n+4 \over
n-4}+{n+4 \over n-4}w^{8 \over n-4}w_t\big].
\end{equation}
Since the volume
$$\int_{S^n}w(t)^{2n \over n-4}d\mu_{S^n}=\int_{S^n}u^{2n
\over n-4}d\mu_{S^n}$$ 
is preserved, it yields that
\begin{equation}\label{vol_t}
\int_{S^n}w^{n+4 \over n-4} w_t d\mu_{S^n}=0.
\end{equation}
By Lemma \ref{conf_inv_P_g}, (\ref{eqn_w_t}) and (\ref{vol_t}), we
have
\begin{eqnarray*}
0&=&{d \over dt}E_{g_{S^n}}[u]={d \over dt}E_{g_{S^n}}[w(t)]=2\int_{S^n}P_{S^n}(w_t)w d\mu_{S^n}\\
&=& (n-4) \int_{S^n}[\xi \cdot d(Q_g \circ \phi)w^{n+4 \over n-4}+{n+4
\over n-4}w^{8 \over n-4}w_t]w d\mu_{S^n}\\
&=&(n-4) \int_{S^n}\xi \cdot d(Q_g \circ \phi) w^{2n \over
n-4}d\mu_{S^n}.
\end{eqnarray*}
Therefore, the desired assertion is followed by the above identity
evaluated at $t=0$.
\end{proof}

\section{Proof of the main result}\label{sec3}

\indent \indent In order to prove the main Theorem \ref{Thm1}, we first need an elementary result.
\begin{lemma}\label{lem1}
For any fixed $\epsilon>0$, then there exist $C_1(n,\epsilon)>0$ and $C_2(n,\epsilon)>0$ such that as $\alpha \to 0^+$, 
 \begin{eqnarray*}
 &&\int_{0}^{\frac{\epsilon}{\alpha}}\Big[1-\tfrac{(n-4)(n^2-4n+8)}{n(n-2)}\tfrac{\sigma^4}{(1+\sigma^2)^2}\Big](1+\sigma^2)^{4-n}\sigma^{n-1}d\sigma\\
 &=&\left\{ \begin{array}{ll}
 -C_1(n,\epsilon), \quad &\hbox{~~if~~} n>8;\\
 C_2(n,\epsilon)\log \alpha, \quad &\hbox{~~if~~} n=8.
 \end{array}
 \right.
 \end{eqnarray*}
\end{lemma}
\begin{proof}
By a direct computation, one has
\begin{eqnarray*}
&&\int \sigma^{n-1}(1+\sigma^2)^{4-n}d\sigma \\
&=& {1\over n}\sigma^n(1+\sigma^2)^{4-n}+{2(n-4)\over n}\int \sigma^{n+1}(1+\sigma^2)^{3-n}d\sigma\\
&=& {1\over n}\sigma^n(1+\sigma^2)^{4-n}+{2(n-4) \over n(n+2)} \sigma^{n+2}(1+\sigma^2)^{3-n}\\
& & +{4(n-4)(n-3)\over n(n+2)}\int \sigma^{n+3} (1+\sigma^2)^{2-n} d\sigma.
\end{eqnarray*}
Thus we conclude that
\begin{eqnarray*}
&&\int_{0}^{\frac{\epsilon}{\alpha}}\Big[1-\tfrac{(n-4)(n^2-4n+8)}{n(n-2)}\tfrac{\sigma^4}{(1+\sigma^2)^2}\Big](1+\sigma^2)^{4-n}\sigma^{n-1}d\sigma\\
&=& {1\over n}(\tfrac{\epsilon}{\alpha})^n(1+(\tfrac{\epsilon}{\alpha})^2)^{4-n}+\tfrac{2(n-4)}{n(n+2)} (\tfrac{\epsilon}{\alpha})^{n+2}(1+(\tfrac{\epsilon}{\alpha})^2)^{3-n}\\
& & +\tfrac{n-4}{ n}\Big(\tfrac{4(n-3)}{n+2}-\tfrac{n^2-4n+8}{n-2}\Big)\int_0^{\tfrac{\epsilon}{\alpha}} \sigma^{n+3} (1+\sigma^2)^{2-n} d\sigma\\
&=& \tfrac{1}{n}(\tfrac{\epsilon}{\alpha})^n(1+(\tfrac{\epsilon}{\alpha})^2)^{4-n}+\tfrac{2(n-4)}{n(n+2)} (\tfrac{\epsilon}{\alpha})^{n+2}(1+(\tfrac{\epsilon}{\alpha})^2)^{3-n}\\
& & -\tfrac{n-4}{n(n+2)(n-2)}\big[(n-8)(n^2+2n+36)+280\big]\int_0^{\epsilon\over\alpha} \sigma^{n+3} (1+\sigma^2)^{2-n} d\sigma.
\end{eqnarray*}
Note that when $n>8$, the first two terms on the right hand side of the last identity go to 0 when $\alpha\to 0$, and the last term has a negative limit, depending only on $n$. When $n=8$, the first two terms have finite limits, but the last term goes to $-\infty$ at the speed of $-\log \alpha$. Therefore, the proof is complete.
\end{proof}

\noindent{\bf The proof of Theorem \ref{Thm1}.~~}
By Theorem 5.1 in \cite{lp}, there exists some conformal metric $g$ in the conformal class of $g_0$ with conformal normal coordinates at $p$, such that $\forall N \geq 5$, in a chart near $p$ there hold
\begin{eqnarray*}
&&\det g=1+O(r^N), \quad r=|x|=\hbox{dist}(x,p),\\
&&R_g=O(r^2), \quad -\Delta_g R_g(p)=\frac{1}{6}|W(p)|^2.
\end{eqnarray*}
Using the above facts and \ref{def_Q-curv}, one obtains
$$Q_g=\frac{|W(P)|^2}{12(n-1)}+O(r).$$

With its corresponding polar coordinates $\{(r, \varphi); \varphi \in S^{n-1} \}$ and polar metric $\tilde{g}$, there hold
$$\tilde{g}_{rr}=1, \quad \tilde{g}_{r\varphi}=0.$$
Thus we obtain
\begin{eqnarray*}
\sqrt{\det g} ~dx^1 \wedge \cdots \wedge dx^n=r^{n-1}\sqrt{\det g} ~dr \wedge d\Omega_{S^{n-1}},
\end{eqnarray*}
where $d\Omega_{S^{n-1}}$ denotes the volume form on $S^{n-1}$. As shown in \cite{gur_mal}, for any radial function $u$, one has
$$\Delta_g u=\Delta_0 u+O(r^{N-1})u',$$
where $\Delta_0$ denotes the Euclidean Laplacian.

Let $\eta_\epsilon \in C_c^\infty(\mathbb{R}^n)$ be a cutoff function for $\epsilon>0$, $\eta_\epsilon=1$ in $B_\epsilon(0)$ and $\eta_\epsilon=0$ outside of $B_{2\epsilon}(0)$. Let
$$u_\alpha(x)=\Big(\frac{2 \alpha}{\alpha^2+|x|^2}\Big)^{\frac{n-4}{2}}, \hbox{~~for any~~} \alpha>0,$$
then,
$$q(S^n)=\frac{n-4}{2}Q_{S^n}\omega_n^{\frac{4}{n}}=\frac{\int_{\mathbb{R}^n}|\Delta_0 u_\alpha|^2 dx}{\Big(\int_{\mathbb{R}^n}u_\alpha^{\frac{2n}{n-4}}dx\Big)^{\frac{n-4}{n}}},$$
where $\omega_n=\hbox{vol}(S^n,g_{S^n})$ denotes the volume of the standard sphere $S^n$. For the Paneitz-Sobolev constant $q(S^n)$, one may refer to Proposition 1.1 in \cite{dhl}.

Notice that
 \begin{eqnarray*}
 &&u_\alpha'=-(n-4)\frac{r}{\alpha^2+r^2}u_\alpha,\quad u_\alpha''=\frac{(n-4)(2r^2-\alpha^2)}{(\alpha^2+r^2)^2}u_\alpha,\\
 &&\Delta_0 u_\alpha=u_\alpha''+\frac{n-1}{r}u_\alpha'=-\frac{n-4}{(\alpha^2+r^2)^2}u_\alpha[(n-3)r^2+n\alpha^2].
 \end{eqnarray*}
 
For any fixed $\epsilon>0$, choose
$$\varphi_\alpha(x)=\eta_\epsilon(x)u_\alpha(x),$$
then
\begin{eqnarray*}
&&\int_{M^n}\varphi_\alpha P_g \varphi_\alpha d\mu_g\\
&=&\int_{M^n}\Big[ |\Delta_g \varphi_\alpha|^2-4A_g(\nabla \varphi_\alpha,\nabla \varphi_\alpha)+(n-2)\sigma_1(A_g)|\nabla \varphi_\alpha|_g^2+\frac{n-4}{2}Q_g \varphi_\alpha^2 \Big]d\mu_g.
\end{eqnarray*}
We start to compute term by term on the right hand side of the above identity.

(i) To estimate
\begin{eqnarray*}
&&\int_{M^n}|\Delta_g \varphi_\alpha|^2 d\mu_g\\
&=&\int_{B_\epsilon(0)}|\Delta_0 u_\alpha+O(r^{N-1})u_\alpha'|^2 (1+O(r^N))dx+\int_{A_\epsilon}|\Delta_g \varphi_\alpha|^2 d\mu_g,
\end{eqnarray*}
where $A_\epsilon=B_{2\epsilon}\setminus B_\epsilon(0)$.

It is not hard to check that
\begin{eqnarray*}
\int_{\mathbb{R}^n \setminus B_\epsilon(0)}|\Delta_0 u_\alpha|^2 dx&=&O\Big(\Big(\frac{\alpha}{\epsilon}\Big)^{n-4}\Big);\\
\int_{A_\epsilon}|\Delta_0 \varphi_\alpha|^2 dx&=&\int_{A_\epsilon}|\Delta_0 \eta_\epsilon u_\alpha + 2 \eta_\epsilon' u_\alpha'+\eta_\epsilon \Delta_0 u_\alpha|^2 dx\\
&\leq& 2 \int_{A_\epsilon}[|\Delta_0 \eta_\epsilon u_\alpha|^2+|\Delta_0 u_\alpha|^2]dx\\
&=&O(\epsilon^{8-n}\alpha^{n-4})+O(\epsilon^{4-n}\alpha^{n-4});\\
\int_{A_\epsilon} |\Delta_g \varphi_\alpha |^2 d\mu_g&=&\int_{A_\epsilon}|\Delta_0 \varphi_\alpha+O(r^{N-1})\varphi_\alpha'|^2 (1+O(r^N)) dx\\
&=&O(\epsilon^{4-n}\alpha^{n-4})+O(\epsilon^{2N-n+4}\alpha^{n-4})
\end{eqnarray*}
Thus, we obtain
$$\int_{M^n}|\Delta_g \varphi_\alpha|^2 d\mu_g=\int_{\mathbb{R}^n}|\Delta_0 u_\alpha|^2 dx+O_\epsilon(\alpha^{n-4}).$$

(ii) To estimate
\begin{eqnarray*}
&&\int_{M^n}Q_g \varphi_\alpha d\mu_g\\
&=&\int_{B_\epsilon(0)}\big(\tfrac{|W(P)|^2}{12(n-1)}+O(r)\big)u_\alpha^2(1+O(r^N))dx+\int_{A_\epsilon}\big(\tfrac{|W(P)|^2}{12(n-1)}+O(r)\big)\varphi_\alpha^2 (1+O(r^N))dx.
\end{eqnarray*}

Thus one has
\begin{eqnarray*}
&&\frac{n-4}{2}\int_{M^n}Q_g \varphi_\alpha d\mu_g\\
&=&\frac{n-4}{24(n-1)}|W(p)|^2\int_{B_\epsilon(0)}u_\alpha^2 dx+\int_{B_{\epsilon}(0)}O(r) u_\alpha^2 dx+\int_{A_\epsilon}O(1)u_\alpha^2 dx\\
&=&\frac{n-4}{24(n-1)}|W(p)|^2\int_{B_\epsilon(0)}u_\alpha^2 dx+O_\epsilon(\alpha^{n-4}).
 \end{eqnarray*}
 
 (iii) To estimate
 \begin{eqnarray*}
 &&-4\int_{B_\epsilon(0)}A_g(\nabla u_\alpha, \nabla u_\alpha)d\mu_g\\
 &=&-4\int_{B_\epsilon(0)}\big(A_{ij}(p)+A_{ij,k}(p)x^k+\tfrac{1}{2}A_{ij,kl}(p)x^kx^l+O(r^3)\big)x^ix^j r^{-2}|u_\alpha'|^2(1+O(r^N))dx\\
&=&-2\int_{B_\epsilon(0)}\big(A_{ij,kl}(p)x^kx^l+O(r^3)\big)x^ix^jr^{-2}|u_\alpha'|^2 dx\\
&=&-\frac{|W(p)|^2}{6n(n-1)(n-2)}\int_{B_\epsilon(0)}r^2 |u_\alpha'|^2 dx+\int_{B_\epsilon(0)}O(r^3)|u_\alpha'|^2 dx.
\end{eqnarray*}
Observe that
\begin{eqnarray*}
&&\int_{B_\epsilon(0)}O(r^3)|u_\alpha'|^2dx\\
&=& \int_0^{\epsilon}O(r^3)\frac{r^2}{(\alpha^2+r^2)^2}\Big(\frac{2\alpha}{\alpha^2+r^2}\Big)^{n-4}r^{n-1}dr\\
&\stackrel{\sigma=\tfrac{\epsilon}{\alpha}}{=}&O(\alpha^5)\Big[O(1)+\int_1^{\tfrac{\epsilon}{\alpha}}\sigma^{8-n}d\sigma\Big]\\
&=&\left\{\begin{array}{lll}
O_\epsilon(\alpha^4), \quad & \hbox{~~if~~} n=8;\\
O_\epsilon(\alpha^5 \log \alpha^{-1}), \quad &\hbox{~~if~~} n=9;\\
O_\epsilon(\alpha^5) \quad &\hbox{~~if~~} n\geq 10.
\end{array}
\right.
\end{eqnarray*}
and
\begin{eqnarray*}
&&\int_{A_\epsilon}A_g(\nabla \varphi_\alpha, \nabla \varphi_\alpha)d\mu_g\\
&=&O_\epsilon (1) \int_{A_\epsilon} \big[|u_\alpha'|^2+|u_\alpha|^2\big]dx\\
&=&O_\epsilon(\alpha^{n-4}).
\end{eqnarray*}
Thus, we obtain
\begin{eqnarray*}
&&-4\int_{B_\epsilon(0)}A_g(\nabla u_\alpha, \nabla u_\alpha)d\mu_g\\
&=&-\frac{|W(p)|^2}{6n(n-1)(n-2)}\int_{B_\epsilon(0)}r^2 |u_\alpha'|^2 dx+\left\{\begin{array}{lll}
O_\epsilon(\alpha^4), \quad & \hbox{~~if~~} n=8;\\
O_\epsilon(\alpha^5 \log \alpha^{-1}), \quad &\hbox{~~if~~} n=9;\\
O_\epsilon(\alpha^5) \quad &\hbox{~~if~~} n\geq 10.
\end{array}
\right.
\end{eqnarray*}
 
 (iv) To estimate
 \begin{eqnarray*}
 &&(n-2)\int_{B_\epsilon(0)}\sigma_1(A_g)|u_\alpha'|^2 d\mu_g\\
 &=&\frac{n-2}{2(n-1)}\int_{B_\epsilon(0)}R_g |u_\alpha'|^2 (1+O(r^N))dx\\
 &=&-\frac{n-2}{24n(n-1)}|W(p)|^2 \int_{B_\epsilon(0)}r^2|u_\alpha'|^2 dx+\int_{B_\epsilon(0)}O(r^3)|u_\alpha'|^2 dx.
 \end{eqnarray*}
 Together with some computations of (iii), we obtain
 \begin{eqnarray*}
 && (n-2)\int_{M^n}\sigma_1(A_g)|u_\alpha'|^2 d\mu_g\\
 &=&-\frac{n-2}{24n(n-1)}|W(p)|^2 \int_{B_\epsilon(0)}r^2|u_\alpha'|^2 dx
 +\left\{\begin{array}{lll}
O_\epsilon(\alpha^4), \quad & \hbox{~~if~~} n=8;\\
O_\epsilon(\alpha^5 \log \alpha^{-1}), \quad &\hbox{~~if~~} n=9;\\
O_\epsilon(\alpha^5), \quad &\hbox{~~if~~} n\geq 10.
\end{array}
\right.
 \end{eqnarray*}
 
 Consequently, we are in a position to compute the coefficient of $|W(p)|^2$ by Lemma \ref{lem1}:
 \begin{eqnarray*}
 &&\frac{n-4}{24(n-1)}\int_{B_\epsilon(0)}u_\alpha^2 dx-\frac{n^2-4n+8}{24n(n-1)(n-2)}\int_{B_\epsilon(0)}r^2 |u_\alpha'|^2 dx\\
&=&\frac{n-4}{24(n-1)}\int_{B_\epsilon(0)}u_\alpha^2\Big[1-\frac{(n-4)(n^2-4n+8)}{n(n-2)}\frac{r^4}{(\alpha^2+r^2)^2}\Big]dx\\
&\stackrel{\sigma=\tfrac{\epsilon}{\alpha}}{=}&\frac{(n-4)2^{n-4}\omega_{n-1}}{24(n-1)}\alpha^4 \int_{0}^{\frac{\epsilon}{\alpha}}\Big[1-\tfrac{(n-4)(n^2-4n+8)}{n(n-2)}\tfrac{\sigma^4}{(1+\sigma^2)^2}\Big](1+\sigma^2)^{4-n}\sigma^{n-1}d\sigma\\
&=&-\left\{ \begin{array}{ll}
 O_\epsilon(\alpha^4), \quad &\hbox{~~if~~} n>8;\\
 O_\epsilon(\alpha^4 \log \alpha^{-1}), \quad &\hbox{~~if~~} n=8.
 \end{array}
 \right.
 \end{eqnarray*}
 
 (v) To estimate
 \begin{eqnarray*}
 &&\int_{M^n}\varphi_\alpha^{\frac{2n}{n-4}}d\mu_g\\
 &=&\int_{\mathbb{R}^n}u_\alpha^{\frac{2n}{n-4}}dx+\Big[\int_{\mathbb{R}^n\setminus B_\epsilon(0)}+\int_{A_\epsilon}\varphi_\alpha^{\frac{2n}{n-4}}(1+O(r^N))dx\Big]\\
 &=&\int_{\mathbb{R}^n}u_\alpha^{\frac{2n}{n-4}}dx+O((\tfrac{\alpha}{\epsilon})^n).
  \end{eqnarray*}
 Thus, we obtain
 $$\Big(\int_{M^n}\varphi_\alpha^{\frac{2n}{n-4}}d\mu_g\Big)^{\frac{n-4}{n}}=\Big(\int_{\mathbb{R}^n}u_\alpha^{\frac{2n}{n-4}}dx\Big)^{\frac{n-4}{n}}+O((\tfrac{\alpha}{\epsilon})^n).$$
 
Therefore, putting these above facts together, we conclude that
\begin{eqnarray*}
q(g)&\leq&\frac{\int_{M^n}\varphi_\alpha P_g \varphi_\alpha d\mu_g}{\Big(\int_{M^n}\varphi_\alpha^{\frac{2n}{n-4}}d\mu_g\Big)^{\frac{n-4}{n}}}\\
&=&Y(S^n)- \left\{\begin{array}{ll}
O_\epsilon(\alpha^4)|W(p)|^2+O_\epsilon(\alpha^5),  &\hbox{~~if~~} n\geq10\\
O_\epsilon(\alpha^4)|W(p)|^2+O_\epsilon(\alpha^5 \log \alpha^{-1}), &\hbox{~~if~~} n=9\\
O_\epsilon(\alpha^4 \log \alpha^{-1})|W(p)|^2+O_\epsilon(\alpha^4), &\hbox{~~if~~} n=8
\end{array}
\right.\\
&<&Y(S^n) 
\end{eqnarray*}
for all $n \geq 8$ by choosing $\alpha$ sufficiently small. \hfill $\Box$

\section{An application to the convergence of the nonlocal $\mathbf{Q}$-curvature flow}\label{sec4}

\indent\indent Recently, Gursky and Malchiodi introduced \cite{gur_mal} a nonlocal $Q$-curvature flow on a closed Riemannian manifold $(M^n,g_0)$ of dimension $n \geq 5$:
\begin{eqnarray}
\frac{\partial u}{\partial t}&=&-u+\mu P_{g_0}^{-1}\big(|u|^{\tfrac{n+4}{n-4}}\big)\label{gm_Qflow}\\
u(0,x)&=&u_0\label{initial data}
\end{eqnarray}
for some initial data $u_0 \in C^\infty_\ast$, where
$$\mu(t)=\frac{\int_{M^n}uP_g u d\mu_{g_0}}{\Big(\int_{M^n}u^{\tfrac{2n}{n-4}}\Big)^{\tfrac{n-4}{n}}}$$
and
$$C^\infty_\ast\equiv\{w \in C^\infty(M^n,g_0); w>0, P_{g_0}w \geq 0\}.$$

Under the assumptions that the $Q_{g_0}$ is semi-positive and the scalar curvature $R_{g_0}$ is nonnegative, which yield that the Paneitz-Sobolev constant $q(g_0)=q(M^n,g_0)$ is positive by Proposition 2.3 in \cite{gur_mal}. The positivity of $u$ is preserved along the nonlocal $Q$-curvature flow and the long time existence of the above nonlocal $\mathbf{Q}$-curvature flow is established. From now on, we adopt the above assumptions in this section. Thus, the $Q$-curvature equation gives
$$P_{g_0}u=Q_g u^{\tfrac{n+4}{n-4}}, \qquad u>0 \hbox{~~on~~} M^n$$
where $Q_g$ is the $Q$-curvature of the flow metric $g(t)=u(t)^{\frac{4}{n-4}}g_0$. 

From Lemma 3.3 in \cite{gur_mal}, $\mu(t)$ is non-increasing and  uniformly bounded below and above by two positive constants, as well as the volume of the flow metric $\int_{M^n}d\mu_{g(t)}$, then it yields
$$\lim_{t \to \infty}\mu(t)=\mu_\infty>0.$$ 
For brevity, set
$$\varphi=-u+\mu P_{g_0}^{-1}\big(|u|^{\tfrac{n+4}{n-4}}\big)$$
and
$$F_2(t)=\int_{M^n}\varphi P_{g_0} \varphi d\mu_{g_0}.$$

In essence, we can further show the asymptotic behavior of $F_2(t)$ as $t \to \infty$.
\begin{lemma}
There holds
$$\lim\limits_{t \to \infty}F_2(t)=0.$$
\end{lemma}
\begin{proof}
By \eqref{gm_Qflow}, a direct computation yields
\begin{eqnarray*}
\frac{1}{2}\frac{d}{dt}F_2(t)&=&\int_{M^n}\varphi P_{g_0}\varphi_t d\mu_{g_0}\\
&=&-\int_{M^n}\varphi P_{g_0}\varphi d\mu_{g_0}+\mu_t\int_{M^n}\varphi u^{\tfrac{n+4}{n-4}}d\mu_{g_0}
+\frac{n+4}{n-4}\mu \int_{M^n}u^{\tfrac{8}{n-4}}\varphi^2 d\mu_{g_0}.
\end{eqnarray*}
By Lemma 3.3 in \cite{gur_mal} and H\"{o}lder's inequality, one has
\begin{eqnarray*}
\Big|\mu_t\int_{M^n}\varphi u^{\tfrac{n+4}{n-4}}d\mu_{g_0}\Big|&\leq& C F_2(t)\Big(\int_{M^n}\varphi^{2n \over n-4}d\mu_{g_0}\Big)^{n-4 \over 2n}\Big(\int_{M^n} u^{2n \over n-4}d \mu_{g_0}\Big)^{n+4 \over 2n}\\
&\leq& C q(g_0)^{-{1 \over 2}} F_2^{3 \over 2}(t).
\end{eqnarray*}
By H\"{o}lder's inequality, we estimate
\begin{eqnarray*}
\Big|\int_{M^n}\alpha f u^{8 \over n-4} \varphi^2 d\mu_{g_0}\Big|&\leq& C\Big(\int_{M^n}u^{2n \over n-4}d\mu_{g_0}\Big)^{4 \over n}\Big(\int_{M^n}\varphi^{2n \over n-4}d\mu_{g_0}\Big)^{n-4 \over n}\\
&\leq& C q(g_0)^{-1} F_2(t).
\end{eqnarray*}
Thus, we obtain
\begin{equation}\label{diff_ineq1}
\frac{d}{dt} F_2(t) \leq CF_2(t)(1+F_2(t)^{1 \over 2}).
\end{equation}
From the estimate $\int_0^\infty F_2(t)dt< \infty$, there exists a sequence $\{t_j\}$ with $t_j \to \infty$ as $j \to \infty$, such that
$$\lim_{t \to \infty}F_2(t_j)=0.$$
Set
$$H(t)=\int_0^{F_2(t)}\frac{ds}{1+s^{1 \over 2}}=2F_2(t)^{\tfrac{1}{2}}-2\log\big(1+F_2(t)^{1 \over 2}\big),$$
then as the same argument in \cite{chxu}, we assert that there exists some uniform constant $C_0>0$ such that  
$$H(t) \geq C_0 F_2(t)$$ 
for sufficiently large $t \geq 0$.   Using 
$\lim\limits_{j \to \infty}H(t_j)=0$ and \eqref{diff_ineq1}, for any $t \geq t_j$, we obtain
$$H(t) \leq H(t_j)+C \int_{t_j}^t F_2(\tau)d\tau.$$
Then, we conclude that 
$$\lim\limits_{t \to \infty}F_2(t) \leq C_0^{-1}\lim\limits_{t \to \infty}H(t)=0.$$
This completes the proof.
\end{proof}

As a direct consequence, we apply Theorem \ref{Thm1} to recover the sequential convergence of the above $Q$-curvature flow in a special case of \cite{gur_mal}.

\begin{corollary}
Let $(M^n,g_0)$ be a closed Riemannian manifold of $n \geq 8$, suppose $M^n$ is not locally conformally flat and Gursky-Malchiodi's assumptions hold true, that is, $Q_{g_0} \geq 0$ and is positive somewhere; $R_{g_0} \geq 0$. Then the nonlocal Q-curvature problem \eqref{gm_Qflow}-\eqref{initial data} is sequentially convergent as $t \to \infty$.
\end{corollary}
\begin{proof}
As the proof of Theorem \eqref{Thm1}, based on test functions in the proof of Therem \ref{Thm1} or \cite{er}, Gursky-Malchiodi set up a scheme to modify them to a sequence $\{\hat{u}^0_n\}$ of positive functions as initial data of the flow. Then by adopting the same argument in Theorem 6.1 of \cite{gur_mal}, the sequential convergence of the nonlocal Q-curvature problem \eqref{gm_Qflow}-\eqref{initial data} follows.
\end{proof}

\end{document}